\pdfoutput=1
\RequirePackage{ifpdf}
\ifpdf 
\documentclass[pdftex]{sigma}
\else
\documentclass{sigma}
\fi

\numberwithin{equation}{section}

\newtheorem{Theorem}{Theorem}[section]
\newtheorem*{Theorem*}{Theorem}
\newtheorem{Corollary}[Theorem]{Corollary}
\newtheorem{Lemma}[Theorem]{Lemma}
\newtheorem{Proposition}[Theorem]{Proposition}
 { \theoremstyle{definition}

\newtheorem{Example}[Theorem]{Example}
\newtheorem{Remark}[Theorem]{Remark} }

\def\Var{\mathrm{Var}}
\def\Nov{\mathrm{Nov}}
\def\Com{\mathrm{Com}}
\def\As{\mathrm{As}}
\def\Lie{\mathrm{Lie}}
\def\Perm{\mathrm{Perm}}
\def\Zinb{\mathrm{Zinb}}
\def\Pois{\mathrm{Pois}}
\def\Jord{\mathrm{Jord}}
\def\SJord{\mathrm{SJord}}
\def\pre{\mathrm{pre}}

\def\wt {\mathop {\fam 0 wt}\nolimits}
\def\id {\mathop {\fam 0 id}\nolimits}
\def\Der {\mathop {\fam 0 Der}\nolimits}
\def\End {\mathop {\fam 0 End}\nolimits}

\begin{document}

\allowdisplaybreaks

\newcommand{\arXivNumber}{2305.07371}

\renewcommand{\PaperNumber}{017}

\FirstPageHeading

\ShortArticleName{On Pre-Novikov Algebras and Derived Zinbiel Variety}

\ArticleName{On Pre-Novikov Algebras and Derived Zinbiel Variety}

\Author{Pavel KOLESNIKOV~$^{\rm a}$, Farukh MASHUROV~$^{\rm b}$ and Bauyrzhan SARTAYEV~$^{\rm cd}$}

\AuthorNameForHeading{P.~Kolesnikov, F.~Mashurov and B.~Sartayev}

\Address{$^{\rm a)}$~Sobolev Institute of Mathematics, Novosibirsk, Russia}
\EmailD{\href{mailto:pavelsk77@gmail.com}{pavelsk77@gmail.com}}

\Address{$^{\rm b)}$~Shenzhen International Center for Mathematics (SICM),\\
\hphantom{$^{\rm b)}$}~Southern University of Science and Technology, Shenzhen, Guangdong, P.R.~China}
\EmailD{\href{mailto:f.mashurov@gmail.com}{f.mashurov@gmail.com}}

\Address{$^{\rm c)}$~Narxoz University, Almaty, Kazakhstan}
\EmailD{\href{mailto:baurjai@gmail.com}{baurjai@gmail.com}}

\Address{$^{\rm d)}$~United Arab Emirates University, Al Ain, United Arab Emirates}

\ArticleDates{Received August 31, 2023, in final form February 05, 2024; Published online February 28, 2024}

\Abstract{For a non-associative algebra $A$ with a derivation $d$, its derived algebra $A^{(d)}$ is the same space equipped with new operations $a\succ b = d(a)b$, $a\prec b = ad(b)$, $a,b\in A$. Given a variety $\Var $ of algebras, its derived variety is generated by all derived algebras~$A^{(d)}$ for all~$A$ in $\Var $ and for all derivations $d$ of $A$. The same terminology is applied to binary operads governing varieties of non-associative algebras. For example, the operad of Novikov algebras is the derived one for the operad of (associative) commutative algebras. We state a sufficient condition for every algebra from a derived variety to be embeddable into an appropriate differential algebra of the corresponding variety. We also find that for $\Var = \Zinb$, the variety of Zinbiel algebras, there exist algebras from the derived variety (which coincides with the class of pre-Novikov algebras) that cannot be embedded into a Zinbiel algebra with a derivation.}


\Classification{17A36; 17A30; 18M60}

\section{Introduction}

The class of nonassociative algebras with one binary operation satisfying the identities of left symmetry
\begin{equation}\label{eq:LSym}
(xy)z - x(yz) = (yx)z - y(xz)
\end{equation}
and right commutativity
\begin{equation}\label{eq:RCom}
(xy)z = (xz)y
\end{equation}
is known as the variety of {\em Novikov algebras}.
Relations \eqref{eq:LSym} and \eqref{eq:RCom} emerged
in \cite{BN1985, GD79} as a~tool
for expressing certain conditions on the components of
a tensor of rank~3 considered
as a collection of structure constants of a finite-dimensional
 algebra with one bilinear operation.
 In~\cite{GD79}, it was a sufficient condition
for a differential operator to be Hamiltonian;
in \cite{BN1985} it was a condition to guarantee
the Jacobi identity for a generalized Poisson bracket in
the framework of Hamiltonian formalism for partial differential
equations of hydrodynamic type.

Novikov algebras may be obtained from (associative) commutative algebras
with a derivation by means of the following operation-transforming functor (see \cite{GD79}).
Assume $A$ is a commutative algebra
with a multiplication~$*$, and let $d$ be a derivation of $A$, i.e.,
 a linear operator $d\colon A\to A$
satisfying the Leibniz rule
\[
d(a*b) = d(a)*b+a*d(b), \qquad a,b\in A.
\]
Then the new operation
\begin{equation}\label{eq:Gelfand}
ab = a*d(b), \qquad a,b\in A,
\end{equation}
satisfies the identities \eqref{eq:LSym} and \eqref{eq:RCom}.

In \cite{DzhLofwall2002}, it was proved that the free Novikov
algebra $\Nov \langle X\rangle $ generated by a set $X$ is
a subalgebra of the free differential commutative algebra
with respect to the operation~\eqref{eq:Gelfand}.
Moreover, it was shown in \cite{BCZ-2017} that every Novikov algebra
can be embedded into an appropriate
commutative algebra with a derivation.

One may generalize the relation between commutative algebras with a derivation and Novikov algebras as follows.
Let $\Var$ be a class of all algebras over a field $\Bbbk $ with one or more bilinear operations which is closed under direct products, subalgebras, and homomorphic images (HSP-class or {\em variety}).
By the classical Birkhoff's theorem, a variety consists of all algebras
that satisfy some family of identities.
We will assume that $\Var $ is defined by multi-linear identities (this is not a restriction if $\operatorname{char}\Bbbk =0$).
For every $A\in \Var $, let $\End(A)$ stand for the set
of all linear operators on the space~$A$.
Then a linear operator $d\in \End(A)$
is called a derivation if the analogue of the Leibniz rule
holds for every binary operation in~$A$.
The set of all derivations of a given algebra $A$ forms a subspace (even a Lie subalgebra) $\Der(A)$ of $\End(A)$.
The class of all pairs $(A,d)$, $A\in \Var$, $d\in \Der(A)$,
is also a variety denoted $\Var\Der$ defined by multi-linear identities.
Given $(A,d)\in \Var\Der$, denote by
$A^{(d)}$ the same space $A$ equipped with two bilinear operations
$\prec$, $\succ $ for each operation on~$A$:
\[
a\prec b = ad(b),\qquad a\succ b = d(a)b, \qquad a,b\in A.
\]
The class of all systems $A^{(d)}$, $(A,d)\in \Var\Der$,
is closed under direct products, so all homomorphic images of
all their subalgebras form a variety denoted $D\Var$,
called the {\em derived variety} of $\Var $ (see \cite{KSO2019}).
Alternatively, $D\Var $ consists of all algebras with duplicated family
of operations that satisfy all those identities that hold on all $A^{(d)}$
for all $(A,d)\in \Var\Der$.

For example, if $\Com $ is the variety of commutative (and associative) algebras then $D\Com = \Nov$ since
$x\prec y = y\succ x$. In general, the description of
$D\Var $ may be obtained in the language of operads and their
Manin products \cite{KSO2019}. Namely, if we identify the notations
for a variety and its governing operad \cite{GK-1994}
then the operad
$D\Var $ coincides with the Manin white product of operads
$\Var $ and $\Nov $.

As a corollary, the free algebra $D\Var \langle X\rangle $
in the variety $D\Var $ generated by a set $X$
is isomorphic to
the subalgebra in the free algebra in the variety
$\Var\Der$ generated by~$X$. However, it is not clear that
the following {\em embedding statement} holds in general:
 every
$D\Var$-algebra can be embedded into an appropriate
differential $\Var$-algebra (or $\Var\Der$-algebra).
In other words, the problem is to determine whether the
class of all subalgebras of all algebras $A^{(d)}$,
$(A,d)\in \Var\Der$,
is closed under homomorphic images.
Positive answers were obtained
for $\Var =\Com $ \cite{BCZ-2017}, $\Var=\Lie$ \cite{KSO2019}, $\Var=\Perm$ \cite{KS2022}, $\Var=\As$ \cite{SK-2022}.

In this paper, we derive a sufficient condition for a positive
answer to the embedding statement for a given $\Var $.
Namely, if the Manin white product $\Nov\circ \Var$ of the corresponding operads coincides with the Hadamard product
$\Nov\otimes \Var $ then the embedding statement holds for $\Var$.

We also find an example of a variety $\Var $ governed by
a binary quadratic operad such that the embedding
statement fails for $\Var$. It turns out that
the variety of Zinbiel algebras $\Zinb $ (also known as commutative dendriform algebras, pre-commutative algebras, dual Leibniz algebras, half-shuffle algebras) introduced in \cite{Loday} provides such an example:
there exists a $D\Zinb $-algebra which cannot be embedded into a Zinbiel algebra with a derivation.
To our knowledge, this is the first example of such a variety $\Var$ that an algebra from $D\Var$ does not embed into a~differential $\Var$-algebra.

The operad $\Zinb $ governing the variety of Zinbiel algebras
is a particular case of a general construction called {\em dendriform splitting} of an operad \cite{BBGN}.
For every binary operad $\Var $ (not necessarily quadratic,
see, e.g., \cite{GubKol2014}) there exists an operad
$\pre\Var $ governing the class of systems with duplicated
family of operations. The generic example of a $\pre\Var$ algebra may be obtained from a $\Var$-algebra
with a Rota--Baxter operator $R$ of weight zero
(see, e.g., \cite{LiGuo_RBABook}).
If $A\in \Var$ and $R\colon A\to A$ is such an operator then
$(A,\vdash, \dashv)$, where
\[
a\vdash b = R(a)b,\qquad a\dashv b = aR(b), \qquad a,b\in A,
\]
is a $\pre\Var $-algebra.
In this context,
$\pre\Com =\Zinb$ (since $a\vdash b = b\dashv a$),
$\pre\Lie$ is the classical variety of left-symmetric algebras (relative to $\vdash $), $\pre\As$ is exactly the variety of dendriform algebras \cite{Loday}.

The theory of pre-algebras and relations between them is similar
in many aspects to the theory of ``ordinary'' algebras.
For example, every $\pre\As $-algebra with respect to the ``dendriform commutator'' $a\vdash b - b\dashv a$ is a $\pre\Lie$-algebra. The properties of the corresponding left adjoint functor (universal envelope) are close to what we have
for ordinary Lie algebras
\cite{DotsTamaroff, Gubarev}.
The class of pre-Novikov algebras has been recently studied in \cite{Hong2023}: it coincides with $D\Zinb $.
Therefore, our results show that the embedding statement
cannot be transferred from ordinary algebras to pre-algebras.

\section{Derived algebras and Manin white product of binary operads}

The details about (symmetric) operads may be found, for example,
in \cite{Bremner-Dotsenko}.
For an operad $\mathcal P$ denote
$\mathcal P(n)$ the linear space (over a base field $\Bbbk $)
of degree $n$ elements of $\mathcal P$, the action of a~permutation $\sigma \in S_n$
on an element $f\in \mathcal P(n)$ is denoted $f^\sigma $,
let $\id \in \mathcal P(1) $ stand for the identity element,
and the composition rule
\[
\mathcal P(n)\otimes \mathcal P(m_1)\otimes \dots \otimes \mathcal P(m_n) \to \mathcal P(m_1+\dots + m_n)
\]
is denoted $\gamma^{m_1+\dots+m_n}_{m_1,\dots , m_n}$.

Recall that an operad $\mathcal P$ is said to be binary, if
$\mathcal P(1)=\Bbbk \mathrm{id}$ and the entire
$\mathcal P$ is generated (as an operad) by its degree 2 space
$\mathcal P(2)$.

Let us fix a binary operad $\mathcal P$.
For every
linear space $A$, consider an operad
$\End_A$
with $\End_A(n)=\mathrm{End}\,(A^{\otimes n}, A)$.
A~morphism of operads
$\mathcal P \to \End_A$ defines an algebra structure
on $A$ with a set of binary operations corresponding
to the generators of $\mathcal P$ from $\mathcal P(2)$.
The class of all such algebras is a variety
 of $\mathcal P$-algebras
defined by multi-linear identities corresponding
to the defining relations of the operad~$\mathcal P$.

Conversely, every variety of algebras with binary operations
defined by multi-linear identities gives rise to
an operad $\mathcal P$ constructed in such a way that
$\mathcal P(n)$ is the space of multi-linear elements
of degree~$n$ in the variables $x_1,\ldots, x_n$
of the free algebra in this variety generated by the countable set
$\{x_1,x_2,\ldots \}$
(see, e.g., \cite[Section~1.3.5]{Kol2008_smz} for details).
In particular, $\id\in \mathcal P(1)$
is presented by the element $x_1$.
Then the variety under consideration consists
exactly of all $\mathcal P$-algebras, in other words, it is
{\em governed} by the operad~$\mathcal P$.

\begin{Example}
Let $\mu,\nu \in \mathcal P(2)$. Suppose we have identified
$\mu $ with $x_1x_2$ and $\nu $ with $x_1*x_2$.
Then{\samepage
\begin{gather*}
\begin{split}
&\gamma^3_{1,2}(\mu, \id, \nu )
= \gamma^3_{1,2} (x_1x_2, x_1, x_1*x_2)
= x_1(x_2*x_3),
\\
&\gamma^4_{2,2}\bigl(\nu^{(12)}, \mu, \nu \bigr)
= \gamma^4_{2,2}(x_2*x_1, x_1x_2, x_1*x_2)
= (x_3*x_4)*(x_1x_2),
\end{split}
\end{gather*}
and so on.}
\end{Example}

We will not distinguish notations for an operad $\mathcal P$
and for the corresponding variety of $\mathcal P$-algebras.

\begin{Example}
Suppose $\mathcal F_2$ is the free binary operad with
$\mathcal F_2(2) \simeq \Bbbk S_2$ (as symmetric modules).
Then the class of $\mathcal F_2$-algebras
coincides with the variety of all nonassociative
algebras with one non-symmetric binary operation,
i.e., $\mathcal F_2(n)$ may be identified with the
linear span of all bracketed monomials
$\bigl(x_{\sigma(1)}\cdots x_{\sigma(n)}\bigr)$, $\sigma \in S_n$,
so $\dim \mathcal F_2(n) = n! C_{n-1}$, where $C_k$ is the $k$th
Catalan number.
\end{Example}

If $\mathcal P$ is an operad governing a variety
of algebras with one binary operation then
$\mathcal P$ is a~homomorphic image of $\mathcal F_2$.
If the kernel of the morphism is generated (as an operadic ideal) by elements from $\mathcal F_2(3)$ then
the operad $\mathcal P$ is said to be {\em quadratic}.
The same definition works for operads with more than one generator.

\begin{Example}\label{exmp:Zinbiel}
Let $\Zinb$ stand for the variety of algebras with one multiplication
satisfying the identity
\begin{equation}\label{eq:Zinbiel}
(x_1\cdot x_2)\cdot x_3
= x_1\cdot (x_2\cdot x_3 + x_3\cdot x_2),
\end{equation}
known as the Zinbiel identity \cite{Loday}.
Then the space $\Zinb(n)$, $n\ge 1$, is spanned by
linearly independent monomials
\[
\bigl( x_{\sigma^{-1}(1)}\cdot \bigl(x_{\sigma^{-1}(2)}\cdot \bigl(\cdots \bigl(x_{\sigma^{-1}(n-1)}\cdot x_{\sigma^{-1}(n)}\bigr)\cdots \bigr)\bigr) \bigr)
=
[x_{1}x_{2}\cdots x_{n-1}x_{n}]^\sigma ,
\qquad \sigma \in S_n,
\]
so $\dim \Zinb(n)=n!$.

The product of two right-normed words in a Zinbiel algebra
can be expressed via shuffle permutations $S_{n,m}\subset S_{n+m}$: if $u= x_1\cdots x_n$, $v = x_{n+1}\cdots x_{n+m}$
then
\[
[u]\cdot [v]
=\sum\limits_{\sigma\in S_{n,m}\colon \sigma (1)=1}
[uv]^\sigma .
\]
The new product $[u]\cdot [v] + [v]\cdot [u]$
on a Zinbiel algebra
is associative and commutative (it is known as the {\em shuffle product}
of words).
\end{Example}

Suppose $\mathcal P_1$ and $\mathcal P_2$ are two binary operads.
Then the {\em Hadamard product} of $\mathcal P_1$ and $\mathcal P_2$ is the operad denoted
$\mathcal P= \mathcal P_1\otimes \mathcal P_2$ such that
$\mathcal P(n)= \mathcal P_1(n)\otimes \mathcal P_2(n)$, $n\ge 1$,
the action of $S_n$ on $\mathcal P(n)$ and the composition rule are defined in
the componentwize way.

The operad $\mathcal P_1\otimes \mathcal P_2$ may not be a binary one (in this paper, we will deal
with such an example below). The sub-operad of $\mathcal P_1\otimes \mathcal P_2$ generated
by $\mathcal P_1(2)\otimes \mathcal P_2(2)$ is
called the {\em Manin white product} of $\mathcal P_1$ and $\mathcal P_2$ denoted by
$\mathcal P_1\circ\mathcal P_2$.
(For some operads $\mathcal P_1$, it may happen that
$\mathcal P_1\circ \mathcal P_2 = \mathcal P_1\otimes \mathcal P_2$
for all $\mathcal P_2$; in \cite{Vallette2008}, all such operads were described.)

If $\mathcal P_1$ and $\mathcal P_2$ are quadratic binary operads then so is
$\mathcal P_1\circ \mathcal P_2$. The defining relations of the last operad can be found
as follows \cite{GK-1994}. Suppose $R_i\subseteq \mathcal P_i(3)$,
$i=1,2$, are the spaces of defining relations of
$\mathcal P_i$ presented in the form of multi-linear identities in $x_1$, $x_2$, $x_3$.
Consider the space~$E(3)$ spanned by all possible compositions of degree~3 of operations from
$\mathcal P_1(2)\otimes \mathcal P_2(2)$.
Then the defining relations of $\mathcal P_1\circ \mathcal P_2$
form the space $E(3)\cap (\mathcal P_1(3)\otimes R_2 + R_1\otimes \mathcal P_2(3))$.

\begin{Example}\label{exmp:Nov-Zinb}
Let $\mathcal P_1=\Nov$, $\mathcal P_2=\Zinb$.
Then $\mathcal P_1(2)\otimes \mathcal P_2(2)$ is spanned by four elements
\begin{alignat*}{3}
& x_1\prec x_2 = x_1x_2\otimes x_1x_2,
 \qquad &&
 x_1\succ x_2 = x_2x_1\otimes x_1x_2, &
 \\
& x_2\prec x_1 = x_2x_1\otimes x_2x_1,
 \qquad &&
 x_2\succ x_1 = x_1x_2\otimes x_2x_1. &
\end{alignat*}
In order to find $E(3)$, calculate all monomials of degree 3 in $x_1$, $x_2$, $x_3$
with operations $\succ$, $\prec$. For example,
\begin{align*}
 (x_1\succ x_3)\prec x_2 &= \gamma^3_{1,2} (x_2\prec x_1, \id, x_1\succ x_2)^{(12)}
 \\
 &=
 \gamma^3_{1,2}(x_2x_1, \id, x_2x_1)^{(12)}\otimes \gamma^3_{1,2}(x_2x_1, \id, x_1x_2)^{(12)}
 \\
& =
(x_3x_2)x_1^{(12)}\otimes (x_2x_3)x_1^{(12)}
=
(x_3x_1)x_2\otimes (x_2x_3)x_1 \in \Nov(3)\otimes \Zinb(3).
\end{align*}
In the same way, the expressions for all 48 monomials may be calculated
in $\Nov(3)\otimes \Zinb(3)$. In order to get defining relations
of $\Nov\circ \Zinb$ it is enough to find those linear combinations
of these monomials that are zero in $\Nov(3)\otimes \Zinb(3)$.
This is a routine problem of linear algebra.
As a result, we obtain the following identities:
\begin{gather}
(x_1\prec x_2)\prec x_3=(x_1\prec x_3)\prec x_2,\nonumber\\
x_1\succ (x_2\succ x_3) = (x_1\succ x_3)\prec x_2-x_1\succ (x_3\prec x_2),\nonumber\\
(x_1\succ x_2)\succ x_3=(x_1\succ x_3)\succ x_2 +(x_1\succ x_3)\prec x_2-(x_1\succ x_2)\prec x_3, \label{eq:pre-Novikov}\\
(x_1\prec x_2)\succ x_3=x_1\prec (x_2\succ x_3)+x_1\prec (x_3\prec x_2)
 +(x_1\succ x_3)\prec x_2-(x_1\prec x_3)\prec x_2.\nonumber
\end{gather}
\end{Example}

\begin{Remark}
Novikov algebras are closely
related to conformal algebras in the sense
of \cite{Kac:VAforbeginners}.
Namely, if $V$ is a Novikov algebra then
the free $\Bbbk [\partial ]$-module $C=\Bbbk[\partial ]\otimes V$
equipped with a~sesqui-linear $\lambda $-product
\[
[u{}_{(\lambda )} v] = \partial (vu) + \lambda (vu + vu),
\qquad u,v\in V,
\]
is a Lie conformal algebra \cite{Xu1999}. The reason of this
relation is that the expression for the generalised Poisson
bracket in \cite{BN1985} has the same form as the expression
for the commutator of local chiral fields in \cite{Kac:VAforbeginners}.
\end{Remark}

\begin{Remark}\label{rem:Conformal}
The variety governed by the operad $\Nov\circ \Zinb $
is also related to conformal algebras.
It is straightforward to check that if $V$ is an algebra over a field $\Bbbk $, $\operatorname{char}\Bbbk =0$, with two
operations~$\prec$,~$\succ $ satisfying the identities
\eqref{eq:pre-Novikov}
then the free
$\Bbbk [\partial ]$-module $C=\Bbbk[\partial ]\otimes V$
equipped with a~sesqui-linear $\lambda $-product
\[
\bigl(u{}_{(\lambda )} v\bigr) = \partial (v\prec u) + \lambda (v\succ u + v\prec u),
\qquad u,v\in V,
\]
is a left-symmetric conformal algebra \cite{HongLi2015}.
We will explain this relation in the last section.
\end{Remark}

Let $\Var $ be a binary operad, we use the same notation for the
corresponding variety of algebras.
Given an algebra $A$ in $\Var $, denote by $\Der(A) $
the set of all derivations of $A$. Recall that
a derivation of $A$ is a linear map $d\colon A\to A$
such that
\[
 d(\mu(a,b)) = \mu(d(a),b) + \mu(a,d(b)), \qquad a,b\in A,
\]
for all operations $\mu $ from $\Var (2)$.
For a derivation $d$ of an algebra $A$, denote by
$A^{(d)}$ the linear space $A$ equipped with {\em derived}
operations
\begin{equation}\label{eq:Der-op}
 \mu^{\succ }(a,b) = \mu(d(a),b), \quad \mu^{\prec }(a,b) = \mu(a,d(b)), \qquad a,b\in A,
\end{equation}
for all $\mu $ in $\Var (2)$.
The variety generated by the class of systems $A^{(d)}$ for all $A\in \Var$ and $d\in \Der(A)$
is denoted by $D\Var$, the {\em derived variety} of $\Var $.
Obviously, the class of all such $A^{(d)}$ is closed under direct products,
so, in order to get $D\Var$, one has to consider all homomorphic images of all subalgebras of these $A^{(d)}$.
In many cases, it is enough to consider just subalgebras.
The problem is to decide whether homomorphic images are actually needed.

\begin{Theorem}[\cite{KSO2019}]
For a binary operad $\Var $, the variety $D\Var $ is governed
by the operad ${\Nov\circ \Var}$.
\end{Theorem}

For example,
all relations that hold
on every Zinbiel algebra with a derivation relative to
the operations $a\succ b = d(a)b$, $a\prec b = ad(b)$
follow from the identities~\eqref{eq:pre-Novikov}.

If $\Var = \Com$ is the variety of associative and commutative algebras
then $D\Var = \Nov$, as follows from the construction of the free Novikov
algebra~\cite{DzhLofwall2002}. Here we have to mention
that commutativity implies $a\succ b = b\prec a$ in every
algebra from $D\Com $.

If $\Var = \Lie $ then the algebras from $D\Var $ form exactly the class of
all $\mathcal F_2$-algebras with one binary operation $a\succ b = -b\prec a$ \cite{KSO2019}.
Note that $\dim \Nov(n) = \binom{2n-2}{n-1}$ \cite{Dzhumad-codimgrowth}, and $\Lie(n)=(n-1)!$.
Hence,
$\dim (\Nov\otimes \Lie)(n) = \frac{(2n-2)!}{(n-1)!}$ which is equal
to $n!C_{n-1}$, where $C_n$ is the $n$th Catalan number.
The number $n!C_{n-1}$ coincides with the $n$th dimension
of the operad~$\mathcal F_2$.
Hence,
$\Nov\circ \Lie = \Nov \otimes \Lie $.

Suppose $\Var $ is a binary operad, $D\Var = \Nov\circ \Var$,
$D\Var\langle X\rangle $ is the free $D\Var$-algebra generated by a countable set
$X=\{x_1,x_2,\dots \}$. Denote by $F=\Var\Der\langle X,d\rangle $ the
free differential $\Var $-algebra generated by $X$ with one derivation $d$.
Then there exists a homomorphism $\tau \colon D\Var\langle X\rangle \to F^{(d)}$
sending $X$ to $X$ identically.
An element from $\ker \tau $ is an identity that holds on all $\Var $-algebras
with a derivation relative to the derived operations \eqref{eq:Der-op}.
Hence, $\tau $ is injective, i.e., the free $D\Var$-algebra
can be embedded into the free differential $\Var $-algebra.

The next question is whether every $D\Var$-algebra can be embedded into an
appropriate differential $\Var $-algebra.
This is the same as to decide if the class of all subalgebras of $A^{(d)}$,
$(A,d)\in \Var\Der$, is closed under homomorphisms.
The answer is positive
for $\Var = \Com $ \cite{BCZ-2017}, $\Lie $ \cite{KSO2019}, $\Perm $ \cite{KS2022}, and $\As $ \cite{SK-2022}.
In the following sections we derive a sufficient condition
for $\Var $ to guarantee a positive answer. This condition is
not necessary, but we find an example when the answer is negative.

\section{The weight criterion and special derived algebras}

Let $\Var $ be a binary operad.
An algebra $V$ with two binary operations $\prec$, $\succ $ from the variety
$D\Var $ is called {\em special} if it can be embedded into
a $\Var $-algebra $A$ with a derivation $d$ such that
$u\prec v = ud(v)$ and $u\succ v = d(u)v$ in $A$ for all
$u,v\in V$.
The class of all $\Var $-algebras with a derivation is a variety
since it is defined by identities. The free differential $\Var $-algebra
$\Var\Der\langle X,d\rangle $
generated by a set $X$ is isomorphic as a $\Var$-algebra
to the free $\Var$-algebra $\Var\bigl\langle X^{(\omega )}\bigr\rangle $
generated by the set
\[
X^{(\omega )} = \bigl\{x^{(n)} \mid x\in X,\, n\in \mathbb Z_+ \bigr\},
\]
with the derivation $d$ defined by $d\bigl(x^{(n)}\bigr) = x^{(n+1)}$,
$x\in X$, $n\in \mathbb Z_+$.

For a nonassociative monomial $u\in X^{(\omega )}$ define its
weight $\wt(u)\in \mathbb Z$ as follows. For a single letter
$u=x^{(n)}$, set $\wt(u)=n-1$. If $u=u_1u_2$ then
$\wt(u) = \wt(u_1)+\wt(u_2)$.
Since the defining relations of $\Var\bigl\langle X^{(\omega )}\bigr\rangle $
are weight-homogeneous, we may define the weight function
on $\Var\Der\langle X,d\rangle $.
Note that if $f\in \Var\bigl\langle X^{(\omega )}\bigr\rangle $
is a weight-homogeneous polynomial then
$\wt d(f) = \wt(f)+1$.

\begin{Lemma}\label{lem:Weight-criterion}
 Let $\Var $ be a binary operad such that
 $\Nov\circ \Var = \Nov \otimes \Var $.
Then for every set~$X$ an element $f\in \Var\bigl\langle X^{(\omega )}\bigr\rangle $
belongs to $D\Var \langle X\rangle $ if and only if
$\wt (f)=-1$.
\end{Lemma}

\begin{proof}
The ``only if'' part of the statement does not depend on the particular operad $\Var $.
Indeed,
every formal expression in the variables $X$ relative to binary operations $\prec$
and $\succ $ turns into a weight-homogeneous polynomial of weight $-1$ in $\Var\bigl\langle X^{(\omega )}\bigr\rangle $.

For the ``if'' part, assume $u$ is a monomial of weight $-1$ in the variables $X^{(\omega )}$.
In the generic form,
\[
 u = \bigl(x_{i_1}^{(s_1)}\cdots x_{i_n}^{(s_n)}\bigr), \qquad x_{i_j}\in X,\quad s_j\ge 0,
\]
with some bracketing. Here $s_1+\dots +s_n = n-1$.
Consider the element
\[
[u] = x_{1}^{(s_1)}\cdots x_{n}^{(s_n)} \otimes (x_1\cdots x_n) \in \Nov(n)\otimes \Var(n).
\]
Here the first tensor factor is a differential commutative monomial of degree $n$ and of weight~$-1$
which belongs to $\Nov(n)$ by \cite{DzhLofwall2002}. In the second factor, we put the
nonassociative multi-linear word
obtained from $u$ by removing all derivatives and consecutive
re-numeration of variables (the bracketing remains the same as in~$u$).

By the assumption, $[u]$ belongs to $(\Nov\circ \Var)(n)$, i.e., can be obtained
from $x_1\prec x_2$ and $x_1\succ x_2$ by compositions and symmetric groups actions.
Hence,
the monomial $\bigl(x_{1}^{(s_1)}\dots x_{n}^{(s_n)}\bigr)$ may be expressed in terms of
operations $\succ $ and $\prec$ on the variables $x_1,\dots,x_n\in X$
in the differential algebra $\Var\bigl\langle X^{(\omega )}\bigr\rangle $.
It remains to make the substitution
$x_{j}\to x_{i_j}$ to get the desired expression for $u$ in $D\Var \langle X\rangle $.
\end{proof}

\begin{Example}
For example, if $u = (x_1 (x_1''x_2 ) ) \in \Lie\bigl\langle X^{(\omega )}\bigr\rangle $
then $[u] = x_1x_2''x_3 \otimes (x_1(x_2x_3))$.
It is straightforward to check that
$[u] = x_1\prec (x_2\succ x_3) - x_2\succ (x_1\prec x_3) - (x_1\prec x_2)\prec x_3$,
where the monomials of degree 3 represent
compositions of $x_1\prec x_2$, $x_2\prec x_1$, $x_1\succ x_2$, $x_2\succ x_1$, and $\id = x_1\otimes x_1$ as in Example~\ref{exmp:Nov-Zinb}.
Hence,
$u = x_1\prec (x_1\succ x_2) - x_1\succ (x_1\prec x_2) - (x_1\prec x_1)\prec x_2$.
\end{Example}

\begin{Remark}
The condition
\[
D\Var \langle X\rangle = \bigl\{f\in \Var\bigl\langle X^{(\omega )}\bigr\rangle \mid \wt (f)=-1 \bigr\}
\]
for a binary operad $\Var $
implies $\Nov \circ \Var =\Nov \otimes \Var $.
Indeed, the one-to-one correspondence
\[
x_1^{(s_1)}\cdots x_n^{(s_n)}\otimes (x_{i_1}\cdots x_{i_n}) \leftrightarrow
\bigl(x_{i_1}^{(s_{i_1})} \cdots x_{i_n}^{(s_{i_n})} \bigr),
\qquad
\sum_i s_i = n-1,
\]
between $(\Nov\otimes \Var)(n)$ and $D\Var (n)$
preserves compositions and symmetric groups
actions. Hence, if
$\bigl(x_{i_1}^{(s_{i_1})} \cdots x_{i_n}^{(s_{i_n})} \bigr)$
may be expressed via $X$ in terms of operations
$\mu^\succ$, $\mu^\prec $ ($\mu \in \Var(2)$)
then
$x_1^{(s_1)}\cdots x_n^{(s_n)}\otimes (x_{i_1}\cdots x_{i_n}) \in (\Nov \circ \Var )(n)$.
\end{Remark}

\begin{Proposition}\label{prop:DNovikov}
The operad $\Var = \Nov $ satisfies the conditions of Lemma~$\ref{lem:Weight-criterion}$, i.e.,
$\Nov\circ \Nov = \Nov\otimes \Nov$.
\end{Proposition}

\begin{proof}
In \cite{Dzhumad-codimgrowth}, a linear basis of the free Novikov algebra generated by an ordered set was described in terms of partitions and Young diagrams. To prove the assertion, we will use this basis, which consists of non-associative monomials constructed from Young diagrams with cells properly filled with generators, see \cite[Section~4]{Dzhumad-Ismailov} for details.

Suppose $h$ is a non-associative monomial of degree $n$
in $\Nov \bigl\langle X^{(\omega )}\bigr\rangle $
of weight $-1$.
The problem is to show that $h\in D\Nov\langle X\rangle $.
Let us proceed by induction both on the degree of $h$
\big(number of letters from $X^{(\omega )}$\big)
and on the number
of ``naked'' letters $x=x^{(0)}$, $x\in X$, that appear in~$h$.
\big(For brevity, letters of the form $x^{(n)}$ for $n>0$ are called ``derived''.\big)

If $\deg h=1$ then $h=x\in X \subset D\Nov\langle X\rangle $.
If $\deg h>1$ but $h$ contains only one ``naked'' letter $x\in X$
then all other letters are of the form $y_i'\in X'$ since
$\wt (h)=-1$. Then the identities of left symmetry \eqref{eq:LSym} and right commutativity \eqref{eq:RCom}
allow us to rewrite $h$ as a linear combination of nonassociative monomials with subwords of the form $xy'$ or $y'x$.
The latter may be processed in a way described below:
e.g., $xy'$ may be replaced with a new letter $(x\prec y)$
so that we get words of smaller degree in the extended alphabet.

{\it Case 1.} If the monomial $h$ has a subword
$a^{(k)}b$ or $ab^{(k)}$
for some $a,b\in X$ and $k\ge 1$,
then
we may transform $h$ to an expression in the extended alphabet (adding a new letter $a\succ b$ to~$X$)~as
\[
a^{(k)}b = (a\succ b)^{(k-1)} - \sum\limits_{s\ge 1} \binom{k-1}{s} a^{(k-s)}b^{(s)},
\]
or, similarly, for $ab^{(k)}$.
The expression in the right-hand side contains monomials either of smaller degree or with a smaller number of ``naked'' letters. Hence,
$h\in D\Nov \langle X\rangle $.

{\it Case 2.}
For the general case,
we need to recall the description of
a linear basis of the free Novikov algebra
(see \cite{Dzhumad-codimgrowth, Dzhumad-Ismailov, DzhLofwall2002}).
Suppose $X^{(\omega )}$ is linearly ordered in such a way
that every ``naked'' letter is smaller than every derived one.

Every element of $\Nov\bigl\langle X^{(\omega)}\bigr\rangle$
may be presented as a linear combination of non-associative words
of the form
\begin{equation}\label{eq:BasisNovikov}
h=(\cdots (W_1W_2) \cdots W_{k-1})W_k,
\end{equation}
where
$W_1 = a_{1,r_1+1}(a_{1,r_1}(\cdots (a_{1,2}a_{1,1})\cdots ))$,
$W_l = a_{l,r_l}(a_{l,r_l-1}(\cdots (a_{l,2}a_{l,1})\cdots ))$,
$l=2,\dots, k$,
$r_1\ge r_2\ge \dots \ge r_k$,
$a_{i,j} \in X^{(\omega )}$.
The letters are ordered in such a way that the following conditions hold:
\begin{itemize}\itemsep=0pt
\item If $r_i=r_{i+1}$ then $a_{i,1}\ge a_{i+1,1}$ for $i=1,\dots, k-1$;
\item
$a_{1,r_1+1}\ge a_{1,r_1}\ge \dots \ge a_{1,2}\ge a_{2,r_2}\ge \dots \ge a_{2,2}\ge \dots \ge a_{k-1,2}\ge a_{k,r_k}\ge \dots \ge a_{k,2}$.
\end{itemize}

In particular, if at least one of the words $W=W_l$ contains both ``naked'' and derived letters then there are two options:
(i) the last letter $a_{l,1}$ is a derived one; (ii) $a_{l,1}$ is ``naked''.
In the first case, the final subword $(a_{l,2}a_{l,1})$ of $W$ is of the form considered in Case~1 since $a_{l,2}$ has to be ``naked'' due to the choice of order on $X^{(\omega )}$. In the second case, we may find a suffix of $W$ which is of the following form:
\[
y^{(n)}(x_1(x_2\cdots (x_{s-1}x_s)\cdots )),\qquad x_i,y\in X,\quad n>0.
\]
An easy induction on $s\ge 1$ shows that the suffix may be transformed (by means of left symmetry) to a sum of monomials
considered in Case~1.

Hence, it remains to consider the case when each $W_l$ contains
either only ``naked'' letters or only derived ones.
Due to the ordering of letters in $X^{(\omega )}$,
the word $h$ of the form \eqref{eq:BasisNovikov}
has the following property:
there exists $1\le l<k$ such that all $W_i$ for $i\le l$ consist of only derived letters and for $i>l$ all $W_i$ are nonassociative words in ``naked'' letters.
Then use right-commutativity to transform $h$ to the form
$h =
(\cdots ((\cdots ((W_1W_{l+1})W_2)\cdots W_{l})W_{l+2}) \cdots W_k)$.
Here $W_1 = y^{(n)}u$, $n>0$, $u$ consists of derived letters,
$W_{l+1} = xv$, $x\in X$, $v$ consists of ``naked'' letters. The subword $W_1W_{l+1}$ may be transformed to
$\bigl(y^{(n)}(xv)\bigr)u$, $n>0$, by right commutativity, and its prefix
$y^{(n)}(xv)$ transforms (by induction on $\deg v$) to a form described in Case~1 by left symmetry:
\[
y^{(n)}(xv) = \bigl(y^{(n)}x\bigr)v - \bigl(xy^{(n)}\bigr)v + x\bigl(y^{(n)}v\bigr).
\tag*{\qed}
\]
\renewcommand{\qed}{}
\end{proof}

\begin{Theorem}\label{thm:Weight-Embedding}
If
$\Var $ is a binary operad such that
$\Var\circ \Nov = \Var\otimes \Nov $
then every $D\Var $-algebra is special.
\end{Theorem}

\begin{proof}
Suppose $V$ is a $D\Var $-algebra.
Then $V$ may be presented as a quotient
of a free algebra $D\Var\langle X\rangle $ modulo an ideal $I$.
Consider the embedding $D\Var\langle X\rangle \subset \Var\bigl\langle X^{(\omega )}\bigr\rangle $
and denote by~$J$ the differential ideal of $\Var\bigl\langle X^{(\omega )}\bigr\rangle $
generated by $I$.
Then $U = \Var \bigl\langle X^{(\omega )}\bigr\rangle /J$ is the universal enveloping
differential $\Var $-algebra of $V$. It remains to prove that $J\cap D\Var\langle X\rangle = I$,
namely, the ``$\subseteq $'' part is not a trivial one.

Assume $f\in J$. Then there exists a family of (differential) polynomials
$\Phi_i \in \Var\bigl\langle (X\cup\{t\})^{(\omega )} \bigr\rangle $
such that $f = \sum_i \Phi_i|_{t=g_i}$, for some $g_i\in I$.
If, in addition, $f\in D\Var\langle X\rangle $ then
$\wt (f)=-1$. Since $\wt g_i=-1$, we should have $\wt \Phi_i=-1$ for all $i$.
By Lemma~\ref{lem:Weight-criterion}, every polynomial $\Phi_i$ may be
represented as an element of $D\Var\langle X\cup\{t\}\rangle $,
so $\Phi_i|_{t=g_i} \in I$ for all $i$, and thus $f\in I$.
\end{proof}

\begin{Corollary}\label{cor:NovDer}
Every $D\Nov$-algebra $V$ with operations $\succ $ and $\prec$
can be embedded
into a commutative algebra with two commuting derivations $d$
and $\partial $
so that
$x\succ y = \partial(x)d(y)$, $x\prec y = x\partial d(y)$
for all $x,y\in V$.
\end{Corollary}

\begin{proof}
For a free $D\Nov $-algebra generated by a set $X$, we have the following
chain of inclusions given by Proposition~\ref{prop:DNovikov}:
\[
D\Nov\langle X\rangle \subset
\Nov\Der\langle X, \partial \rangle = \Nov \bigl\langle X^{(\omega )}\bigr\rangle
\subset \Com\Der\bigl\langle X^{(\omega )}, d \bigr\rangle = \Com \bigl\langle X^{(\omega , \omega )} \bigr\rangle .
\]
Here $X^{(\omega ,\omega )} = \bigl(X^{(\omega )}\bigr)^{(\omega )} = \bigl\{x^{(n,m)} \mid
x\in X, n,m\in \mathbb Z_+ \bigr\}$,
a variable $x^{(n,m)}$ represents $d^n\partial^m(x)$. The elements of $D\Nov \langle X\rangle $
are exactly those polynomials in $\Com \bigl\langle X^{(\omega , \omega )} \bigr\rangle$
that can be presented as linear combinations of monomials
\[
x_1^{(n_1,m_1)}\cdots x_k^{(n_k,m_k)},\qquad \sum_i n_i = \sum_i m_i = k+1.
\]
The same arguments as in the proof of Theorem~\ref{thm:Weight-Embedding} imply
the claim.
\end{proof}

Apart from the operads $\Com $
and $\Lie $ considered above, the operads
$\Pois$ and $\As$ governing
the varieties of Poisson and associative algebras,
respectively, also satisfy the conditions of Theorem~\ref{thm:Weight-Embedding}
\cite{KSO2019, SK-2022}.
However, even if $\Nov\circ \Var \ne \Nov\otimes\Var$
then it is still possible
that every $D\Var $-algebra is special.
For example, if $\Var = \Jord$ is the variety
of Jordan algebras then the corresponding operad
is not quadratic and, in particular,
the element
$[u]= x_1x_2'x_3'\otimes x_1(x_2x_3) \in \Nov(3)\otimes \Jord(3)$ does not belong to $(\Nov\circ \Jord)(3)$.
The operad $\Nov\circ \Jord $ is generated by a
single operation $x_1\succ x_2 = x_2\prec x_1$
due to commutativity of $\Jord $.
Hence, $\Nov\circ \Jord$ is a~homomorphic
image of the free operad~$\mathcal F_2$.
On the other hand, we have

\begin{Proposition}\label{prop:DerAComm}
For every non-associative algebra $V$ with
a multiplication $\nu \colon V\otimes V \to V$
there exists an associative algebra $(A,\cdot )$
with a derivation $d$ such that
$V\subseteq A$ and $\nu (u,v) = d(u)\cdot v + v\cdot d(u)$
for all $u,v\in V$.
\end{Proposition}

In other words, we are going to show that every
non-associative algebra $V$ embeds into the derived
anti-commutator algebra \smash{$\bigl(A^{(+)}\bigr)^{(d)}$} for an appropriate
associative differential algebra $(A,d)$.

\begin{proof}
Let us choose a linear basis $B$ of $V$
equipped with an arbitrary total order $\le $
such that $(B,\le )$ is a well-ordered set.
Then define $F$ to be the free associative algebra
generated by $B^{(\omega )}$.
Induce the order $\le $ on $B^{(\omega )}$
by the following rule:
\[
a^{(n)} \le b^{(m)} \iff (n,a)\le (m,b)\ \text{lexicographically},
\]
and expand it to the words in $B^{(\omega )}$ by the deg-lex rule
(first by length, then lexicographically).

Consider the set of defining relations
\begin{gather*}
S = \bigg\{a^{(n)}b + \sum\limits_{s\ge 1} \binom{n-1}{s}\!
\bigl(a^{(n-s)}b^{(s)} +b^{(s)}a^{(n-s)} \bigr)
+ba^{(n)} - \nu(a,b)^{(n-1)} \mid a,b\in B, n\ge 1\bigg \}.
\end{gather*}
All relations in the set $S$ are obtained
from $a'b+ba'-\nu(a,b)$ by formal derivation
$d\colon x^{(s)}\to x^{(s+1)}$, $x^{(s)}\in B^{(\omega )}$.
Hence, $A=F/(S)$ is a differential associative algebra,
and the map $\varphi \colon V\to A^{(+)}$, $\varphi( v)= v+(S)$, $v\in V$,
preserves the operation, i.e.,
$\varphi(\nu(u,v)) =
d(\varphi(u))\cdot \varphi (v) +\varphi(v)\cdot d(\varphi(u))$ for all $u,v\in V$.

The principal parts of $f\in S$ relative to the order
$\le $ are $a^{(n)}b$, $a,b\in B$, $n\ge 1$.
These words have no compositions of inclusion or intersection, hence, $S$ is a Gr\"obner--Shirshov
basis in $F$ and the images of all variables from $B$ are linearly independent in $A$ since they are $S$-reduced (see, e.g.,~\cite{BC-bull2014} for the
definitions).
Therefore, $\varphi \colon V\to A^{(+)}$ is the desired embedding.
\end{proof}

Consider the variety $\SJord$ generated by all special Jordan algebras (i.e., embeddable into associative ones with respect
to the anti-commutator).
In particular, for every associative algebra~$(A,\cdot)$ with a derivation $d$
the same space $A$ equipped with a new operation
$\nu(u,v) = d(u)\cdot v + v\cdot d(u)$, $u,v\in A$,
is an algebra from $D\SJord$.
Proposition~\ref{prop:DerAComm} implies that
there are no identities that hold for the binary operation $\nu $
like that. Hence, the varieties $D\SJord = D\Jord$ coincide
with the variety of all nonassociative algebras with one operation.
However, again from Proposition~\ref{prop:DerAComm} every
$D\Jord $-algebra embeds into an appropriate Jordan algebra
(even a special Jordan algebra)
with a derivation.

In the next section, we find an example of a variety $\Var $
for which the embedding statement fails.

\section{Dendriform splitting and a non-special pre-Novikov algebra}

Another example of a variety
$\Var $ not satisfying the conditions
of Theorem~\ref{thm:Weight-Embedding}
is the class $\Zinb$ of Zinbiel (dual Leibniz or pre-commutative) algebras.
This is a particular case of the dendriform
splitting of a binary operad described in
\cite{BBGN, GubKol2014}.
Namely,
if $\Var $ is a variety of algebras with (one or more) binary operation
$\mu (x,y) = xy$ satisfying a family
of multi-linear identities $\Sigma $
then
$\pre\Var $ is a variety of algebras with
duplicated set of binary operations
$\mu_\vdash (x,y)=x\vdash y $,
$\mu_\dashv (x,y) = x\dashv y$
satisfying a set of identities $\pre\Sigma $
defined as follows.
Assume $f = f(x_1,\dots, x_n)$ is a multi-linear polynomial of degree $\le n$, and let $k\in \{1,\dots, n\}$.
Suppose $u$ is a nonassociative monomial in the variables
$x_1,\dots, x_n$ such that each $x_i$ appears in $u$
no more than once.
Define a polynomial~$u^{[k]}$ in~$x_1,\dots, x_n$
relative to the operations $\mu_\vdash $, $\mu _\dashv$
by induction on the degree.
If $u=x_i$ then
$u^{[k]} = x_i$;
if $u = vw$ and $x_k$ appears in $v$ (or in $w$) then
$u^{[k]} = v^{[k]}\dashv w^{[k]}$ (or, respectively,
$v^{[k]}\vdash w^{[k]}$);
if $x_k$ does not appear in $u$ then set
$u^{[k]}=
v^{[k]}\dashv w^{[k]} + v^{[k]}\vdash w^{[k]}$.
Transforming each monomial $u$ in the
multi-linear polynomial $f$
in this way, we get $f^{[k]}(x_1,\dots, x_n)$.
The collection of all such $f^{[k]}$
for $f\in \Sigma $, $k=1,\dots,\deg f$,
forms the set of defining relations of a new variety denoted
$\pre\Var $.

For example, for $f(x_1,x_2,x_3)
= (x_1x_2)x_3 - x_1(x_2x_3)$
the polynomials $f^{[k]}$, $k=1,2,3$, are given by
\begin{gather}
f^{[1]} = (x_1\dashv x_2)\dashv x_3 - x_1\dashv (x_2\dashv x_3 + x_2\vdash x_3),\nonumber \\
f^{[2]}
= (x_1\vdash x_2)\dashv x_3 - x_1\vdash (x_2\dashv x_3),\nonumber\\
f^{[3]}
= (x_1\vdash x_2 + x_1\dashv x_2)\vdash x_3 - x_1\vdash (x_2\vdash x_3).\label{eq:preAss}
\end{gather}
These identities define the variety of pre-associative
or dendriform algebras \cite{Loday}.

If the initial operation was commutative or anti-commutative then the set of identities $\pre\Sigma $
includes $x_1\vdash x_2 = \pm x_2\dashv x_1$,
so the operations in $\pre\Var $ are actually
expressed via $\mu_\vdash $ or $\mu_\dashv $.
For example, $\Var=\Lie$ produces the variety
$\pre\Lie$ of left- or right-symmetric algebras
(depending on the choice of $\vdash $ or $\dashv$).
If $\Var = \Com$ then, in terms of the operation
$x\cdot y = x\dashv y = y\vdash x$,
all three identities \eqref{eq:preAss}
of pre-associative algebras are equivalent to~\eqref{eq:Zinbiel}.

In a similar way, one may derive the identities
of a $\pre\Nov$-algebra by means of the dendriform
splitting applied to \eqref{eq:LSym} and \eqref{eq:RCom}.
Routine simplification leads us to the following
definition: a $\pre\Nov$-algebra is a linear space with
two bilinear operations $\vdash $, $\dashv$ satisfying
\begin{gather}
 (x_1\dashv x_2)\dashv x_3 = (x_1\dashv x_3)\dashv x_2, \nonumber\\
 (x_1\vdash x_2)\dashv x_3 = (x_1\vdash x_3)\vdash x_2 + (x_1\dashv x_3)\vdash x_2, \nonumber\\
 (x_1\dashv x_2)\dashv x_3 - x_1\dashv (x_2\dashv x_3) - x_1\dashv (x_2\vdash x_3)
 =
 (x_2\vdash x_1)\dashv x_3 - x_2\vdash (x_1\dashv x_3),
 \nonumber\\
 (x_1\vdash x_3)\dashv x_2 - x_1\vdash (x_2\vdash x_3)
 =
 (x_2\vdash x_3)\dashv x_1 - x_2\vdash (x_1\vdash x_3).\label{eq:pre-Novikov-2}
\end{gather}

The formal change of operations $x\dashv y = x\prec y$, $x\vdash y = y\succ x$
turns \eqref{eq:pre-Novikov-2} exactly into \eqref{eq:pre-Novikov}.
Hence, the operad $\pre\Nov = \pre D\Com$ defines
the same class of algebras as $D\Zinb = D\pre\Com $.

\begin{Remark}
This is not hard to compute that
$\pre D\Lie = D\pre\Lie$ and $\pre D\As = D\pre\As$.
In general, for every binary operad $\Var $ there exists
a morphism of operads $\pre D\Var \to D\pre\Var $
(i.e., every $D\pre\Var $-algebra is a $\pre D\Var$-algebra).
We do not know an example when
this morphism is not an isomorphism, i.e., when the
operations $\pre$ and $D$ applied to a binary operad
do not commute.
\end{Remark}

An equivalent way to define the variety
$\pre\Var $ was proposed in \cite{GubKol2014}.
Let $\Perm $ stand for the variety of associative
algebras that satisfy left commutativity
\[
x_1x_2x_3 = x_2x_1x_3.
\]
An algebra $V$ with two operations $\dashv$, $\vdash $
is a $\pre\Var$-algebra if and only if
for every $P\in \Perm $ the space
$P\otimes V$ equipped with the single operation
\begin{equation}\label{eq:Dend-Perm}
(p\otimes u)(q\otimes v) = pq\otimes (u\vdash v) +
qp\otimes (u\dashv v),
\qquad
p,q\in P,\quad u,v\in V,
\end{equation}
is a $\Var $-algebra.
The same statement holds in the case when the binary operad
$\Var $ is generated by several operations.

\begin{Remark}
For an arbitrary binary operad,
there is a morphism of operads
$\zeta\colon \pre\Var \to \Zinb\circ \Var $. Namely,
for every $A\in \Var $ and for every $Z\in \Zinb$
the space $Z\otimes A$ equipped with two operations
\[
(z\otimes a)\vdash (w\otimes b) = (w\cdot z)\otimes ab,
\qquad
(z\otimes a)\dashv (w\otimes b) = (z\cdot w)\otimes ab,
\]
for $z,w\in Z$, $a,b\in A$, is a $\pre\Var $-algebra.

However, $\pre\Var $ and $\Zinb\circ \Var $ are
not necessarily isomorphic. For example,
if $\Var $ is defined by the identity $(x_1\cdot x_2)\cdot x_3=0$
then the kernel of $\zeta $ is nonzero.
\end{Remark}

As a corollary, we obtain

\begin{Proposition}
The operad $\pre\Nov $
is isomorphic to the Manin white product $\Zinb\circ \Nov$.
\end{Proposition}

\begin{Remark}
Let $V$ be a $D\Zinb$-algebra.
In terms of pre-Novikov operations $\dashv$ and $\vdash$, the conformal algebra structure mentioned in
Remark~\ref{rem:Conformal} is expressed as
\[
(u_{(\lambda)} v) = \partial (v\dashv u) + \lambda (u\vdash v + v\dashv u),
\qquad u,v\in V.
\]
This is indeed a left-symmetric conformal algebra which is easy
to check via the conformal analogue of \eqref{eq:Dend-Perm}.
By slight abuse of notations,
for every $\Perm $-algebra $P$ the operation
\begin{gather*}
[(p\otimes u)_{(\lambda) } (q\otimes v)]
= pq \otimes (u_{(\lambda )} v) - qp\otimes (v_{(-\partial-\lambda )} u) \\
\qquad{}=pq\otimes (\partial (v\dashv u) + \lambda (u\vdash v + v\dashv u))
-qp\otimes (\partial (u\dashv v) - (\lambda+\partial) (v\vdash u + u\dashv v))\\
\qquad{}=\partial (qp\otimes v\vdash u+pq\otimes v\dashv u)
+\lambda (pq\otimes u\vdash v +qp\otimes u\dashv v
+ qp\otimes v\vdash u + pq\otimes v\dashv u)
\end{gather*}
is exactly the quadratic Lie conformal algebra structure \cite{Xu1999}
on $\Bbbk [\partial ]\otimes P\otimes V$
corresponding to the Novikov algebra $P\otimes V$:
\[
[x_{(\lambda )} y] = \partial (y x) + \lambda (x y + y x)
\]
for $x=p\otimes u$, $y=q\otimes v$, and the product
is given by \eqref{eq:Dend-Perm}.

Hence the construction of a $\pre\Lie$ conformal algebra from a
$D\Zinb $-algebra is a quite clear consequence of
the commutativity of tensor product.
\end{Remark}

The final statement of this section shows a substantial difference
between the properties of Novikov algebras and $\pre\Nov $-algebras.
Although the defining identities of $\pre\Nov $
are exactly those that hold on differential Zinbiel algebras
with operations \eqref{eq:Der-op}
(i.e.,
the dendriform analogue of \cite[Theorem~7.8]{DzhLofwall2002}
holds),
the general embedding statement (i.e., the dendriform analogue
of \cite[Theorem~3]{BCZ-2017})
turns to be wrong.

\begin{Theorem}
If the characteristic of the base field $\Bbbk $ is not~$2$ or~$3$
then there exists a $D\Zinb $-algebra which cannot be embedded into a differential Zinbiel algebra.
\end{Theorem}

\begin{proof}
Consider the free Zinbiel algebra $F$ generated
by
\[
\{a,b\}^{(\omega )} = \bigl\{a,b,a',b', \dots, a^{(n)}, b^{(n)}, \dots \bigr\}.
\]
This is the free differential Zinbiel algebra
with two
generators $a$, $b$, its derivation $d$ maps $x^{(n)}$ to $x^{(n+1)}$
for $x\in \{a,b\}$. The product of two elements $f,g\in F$ is denoted
$f\cdot g$.

For every $f,g\in F$, define $f\prec g$, $f\succ g$
by the rule \eqref{eq:Der-op}:
\[
f\prec g = f\cdot d(g),\qquad f\succ g = d(f)\cdot g.
\]
Then $(F,\prec, \succ )$ is a $D\Zinb $-algebra, and its
subalgebra generated by $a$, $b$ is isomorphic to the free algebra $D\Zinb \langle a,b\rangle $.

Denote $f=b\prec b=b\cdot b' \in D\Zinb \langle a,b\rangle \subset F$, and let $J$ stand for the ideal in $F$ generated by~$f$ and all its derivatives:
\[
J = \bigl(f,f',f'', \dots \bigr) \triangleleft F,\qquad d(J)\subseteq J.
\]
In particular,
\[
h = a\cdot \bigl(f'\cdot b'\bigr)-a\cdot \bigl(f\cdot b''\bigr) \in J.
\]
Let us show that $h\in D\Zinb \langle a,b\rangle \subset J$.
Indeed,
\begin{align*}
h &= a\cdot \bigl(\bigl(b\cdot b'\bigr)'\cdot b'\bigr)-a\cdot \bigl(\bigl(b\cdot b'\bigr)\cdot b''\bigr) \\
 &= a\cdot \bigl(\bigl(b'\cdot b'\bigr)\cdot b'\bigr) + a\cdot \bigl(\bigl(b\cdot b''\bigr)\cdot b'\bigr)-a\cdot \bigl(\bigl(b\cdot b'\bigr)\cdot b''\bigr)
 = a\cdot \bigl(\bigl(b'\cdot b'\bigr)\cdot b'\bigr)
\end{align*}
due to the right commutativity of Zinbiel algebras.
Next,
$(b'\cdot b')\cdot b' = 2b'\cdot (b'\cdot b')$,
so
$(b'\cdot b')\cdot b'+b'\cdot (b'\cdot b')
= \frac{3}{2} (b'\cdot b')\cdot b'$.
Therefore,
\begin{align*}
h &= a\cdot \bigl(\bigl(b'\cdot b'\bigr)\cdot b'\bigr)
= \frac{2}{3} a\cdot \bigl(\bigl(b'\cdot b'\bigr)\cdot b' + b'\cdot \bigl(b'\cdot b'\bigr)\bigr)
 = \frac{2}{3} \bigl(a\cdot \bigl(b'\cdot b'\bigr)\bigr)\cdot b'
 \\
 &= \frac{1}{3} \bigl(\bigl(a\cdot b'\bigr)\cdot b'\bigr)\cdot b'= \frac{1}{3} ((a\prec b)\prec b)\prec b = 2\bigl[ab'b'b'\bigr]
 \in D\Zinb \langle a,b\rangle .
\end{align*}
As in Example~\ref{exmp:Zinbiel}, we denote by
$[x_1x_2\cdots x_{n-1}x_n]$ the following expression
in a Zinbiel algebra:
\[
[x_1x_2\cdots x_{n-1}x_n]
= x_1\cdot (x_2\cdot (\cdots (x_{n-1}\cdot x_n)\cdots )).
\]
Recall \cite{Loday} that all such expressions with $x_i$
from a set $X$ form a linear basis of the free Zinbiel algebra
generated by $X$ (i.e., this is a normal form in $\pre\Com\langle X\rangle $).

Let $I$ be the ideal in $D\Zinb \langle a,b\rangle $
generated by $f$, and let $V = D\Zinb \langle a,b\mid f\rangle
=D\Zinb \langle a,b\rangle/I $.
Then
$F/J$ is the universal differential Zinbiel envelope
of $V$. To prove the theorem, it remains to show that $h\notin I$:
if so then $h+I$
would lie in the kernel of every homomorphism
from $V$ to the derived algebra $Z^{(d)}$
constructed
from a differential Zinbiel algebra $Z$ with a derivation~$d$.

Assume $[ab'b'b']\in I$.
The specific of the Zinbiel identity \eqref{eq:Zinbiel}
is that the first letter remains unchanged in all terms.
Hence, $[ab'b'b']$ should be a
linear combination of the elements
$(a\ast f\star b)$, $(a\ast b\star f)$,
where $\ast, \star \in \{\prec, \succ \}$,
with two possible bracketing each, so we have in total 16 terms
under consideration. Let us write them all in the normal form in $F$:
\begin{gather*}
(a\prec f)\prec b
 = 3\bigl[ab'b'b'\bigr]+\bigl[ab'bb''\bigr]+\bigl[abb'b''\bigr]+\bigl[abb''b'\bigr], \\
(a\prec f)\succ b
 =\bigl[a'bb'b'\bigr]+\bigl[a'b'bb'\bigr]+\bigl[a'b'b'b\bigr]+2\bigl[a'bbb''\bigr]+\bigl[a'bb''b\bigr]+\bigl[abb''b'\bigr]+\bigl[ab''bb'\bigr]\\
\hphantom{(a\prec f)\succ b=}{} +\bigl[ab''b'b\bigr]+2\bigl[abb'b''\bigr]+2\bigl[ab'bb''\bigr]+2\bigl[ab'b''b\bigr]+2\bigl[abbb'''\bigr]+\bigl[abb'''b\bigr], \\
(a\succ f)\prec b
 = \bigl[a'b'bb'\bigr]+2\bigl[a'bb'b'\bigr], \\
(a\succ f)\succ b
 = 2\bigl[a''bbb'\bigr]+\bigl[a''bb'b\bigr] +\bigl[a'bb'b'\bigr]+\bigl[a'b'bb'\bigr]+\bigl[a'b'b'b\bigr]+2\bigl[a'bbb''\bigr]+\bigl[a'bb''b\bigr], \\
a\prec (f\prec b)
 = 2\bigl[ab'b'b'\bigr]+2\bigl[abb''b'\bigr]+2\bigl[abb'b''\bigr], \\
a\prec (f\succ b)
 = a\bigl[b'b'b\bigr]'+a\bigl[b'bb'\bigr]' + a\bigl[bb''b\bigr]' + a\bigl[bbb''\bigr]'
 = \bigl[ab''b'b\bigr] + 2\bigl[ab'b''b\bigr]+ 2\bigl[ab'b'b'\bigr]\\
\hphantom{a\prec (f\succ b)=}{} + \bigl[ab''bb'\bigr]+2\bigl[ab'bb''\bigr] +\bigl[abb'''b\bigr]+\bigl[abb''b'\bigr]+\bigl[abb'b''\bigr]+\bigl[abbb'''\bigr], \\
a\succ (f\prec b)
 = 2\bigl[a'bb'b'\bigr], \\
a\succ (f\succ b)
 = \bigl[a'b'b'b\bigr]+\bigl[a'b'bb'\bigr]+\bigl[a'bb''b\bigr]+\bigl[a'bbb''\bigr], \\
(a\prec b)\prec f
 = 3\bigl[ab'b'b'\bigr] + \bigl[ab'bb''\bigr]+\bigl[abb'b''\bigr]+\bigl[abb''b'\bigr], \\
(a\prec b)\succ f
 = \bigl[a'b'bb'\bigr]+2\bigl[a'bb'b'\bigr] + \bigl[ab''bb'\bigr]+\bigl[abb''b'\bigr]+\bigl[abb'b''\bigr], \\
(a\succ b)\prec f
 = \bigl[a'bb'b'\bigr]+\bigl[a'b'bb'\bigr]+\bigl[a'b'b'b\bigr] + 2\bigl[a'bbb''\bigr]+\bigl[a'bb''b\bigr], \\
(a\succ b)\succ f
 = 2\bigl[a''bbb'\bigr]+\bigl[a''bb'b\bigr]+\bigl[a'b'bb'\bigr]+2\bigl[a'bb'b'\bigr], \\
a\prec (b\prec f)
 = \bigl[ab'b'b'\bigr]+\bigl[abb''b'\bigr] +2\bigl[abb'b''\bigr] +\bigl[ab'bb''\bigr] + \bigl[abbb'''\bigr], \\
a\prec (b\succ f)
 = \bigl[ab''bb'\bigr]+\bigl[ab'b'b'\bigr]+\bigl[ab'bb''\bigr], \\
a\succ (b\prec f)
 = \bigl[a'bb'b'\bigr] + \bigl[a'bbb''\bigr], \\
a\succ (b\succ f)
 = \bigl[a'b'bb'\bigr].
\end{gather*}
Arrange the normal Zinbiel words in the following order:
$[ab'b'b']$, $[ab'bb'']$, $[abb'b'']$, $[abb''b']$, $[ab''bb']$, $[ab''b'b]$, $[ab'b''b]$, $[abbb''']$, $[abb'''b]$, $[a''bb'b]$, $[a''bbb']$, $[a'bb'b']$, $[a'b'bb']$, $[a'b'b'b]$, $[a'bbb'']$, $[a'bb''b]$.
Then the assumption $[ab'b'b']\in I$ is equivalent to the condition
that the row vector $e_1 = (1,0,\dots ,0)\in \Bbbk ^{16}$
belongs to the row space of the matrix
\[ \left(
\begin{array}{cccccccccccccccc}
3& 1& 1& 1& 0& 0& 0& 0& 0& 0& 0& 0& 0& 0& 0& 0\\
0& 2& 2& 1& 1& 1& 2& 2& 1& 0& 0& 1& 1& 1& 2& 1\\
0& 0& 0& 0& 0& 0& 0& 0& 0& 0& 0& 2& 1& 0& 0& 0\\
0& 0& 0& 0& 0& 0& 0& 0& 0& 1& 2& 1& 1& 1& 2& 1\\
2& 0& 2& 2& 0& 0& 0& 0& 0& 0& 0& 0& 0& 0& 0& 0\\
2& 2& 1& 1& 1& 1& 2& 1& 1& 0& 0& 0& 0& 0& 0& 0\\
0& 0& 0& 0& 0& 0& 0& 0& 0& 0& 0& 2& 0& 0& 0& 0\\
0& 0& 0& 0& 0& 0& 0& 0& 0& 0& 0& 0& 1& 1& 1& 1\\
3& 1& 1& 1& 0& 0& 0& 0& 0& 0& 0& 0& 0& 0& 0& 0\\
0& 0& 1& 1& 1& 0& 0& 0& 0& 0& 0& 2& 1& 0& 0& 0\\
0& 0& 0& 0& 0& 0& 0& 0& 0& 0& 0& 1& 1& 1& 2& 1\\
0& 0& 0& 0& 0& 0& 0& 0& 0& 1& 2& 2& 1& 0& 0& 0\\
1& 1& 2& 1& 0& 0& 0& 1& 0& 0& 0& 0& 0& 0& 0& 0\\
1& 1& 0& 0& 1& 0& 0& 0& 0& 0& 0& 0& 0& 0& 0& 0\\
0& 0& 0& 0& 0& 0& 0& 0& 0& 0& 0& 1& 0& 0& 1& 0\\
0& 0& 0& 0& 0& 0& 0& 0& 0& 0& 0& 0& 1& 0& 0& 0
\end{array} \right).
\]
This matrix may be transformed by elementary row and column
transformations with integer coefficients
to the triangular form with $1$, $2$, or $0$ on the diagonal,
so that its rank equals 10 for $\operatorname{char}\Bbbk \ne 2$.
Adding one more row $e_1$ increases the rank, so $[ab'b'b']\notin I$.
\end{proof}

As a corollary, we obtain the following observation which is in some sense
converse to Corollary~\ref{cor:NovDer}.
Suppose $W$ is a Novikov algebra equipped with a Rota--Baxter operator, that is, $R\colon W\to W$ is a linear operator
such that
\[
R(u)R(v) = R(uR(v) + R(u) v), \qquad u,v\in W.
\]
Then, in general, there is no commutative algebra
$(A,* )$ with a derivation $d$ and a Rota--Baxter operator $\rho $
such that $W\subseteq A$,
$u v = u* d(v)$, $R(u)=\rho(u)$, for $u,v\in W$,
and $\rho d=d\rho $. In other words,
a Novikov Rota--Baxter algebra cannot be in general embedded
into a commutative Rota--Baxter algebra with a derivation.

Indeed, assume such a system $(A,*, d, \rho )$ exists
for every Novikov algebra with a Rota--Baxter operator.
Every pre-Novikov algebra $V$ with
operations $\vdash $ and $\dashv $
can be embedded into a
Novikov algebra $W=\widehat V$ with a Rota--Baxter operator
so that $u\vdash v = R(u)v$, $u\dashv v = u R(v)$,
$u,v\in V$ (see, e.g., \cite{GubKol2014}).
We may further embed this $W$ into a differential
commutative Rota--Baxter algebra $(A,*, d, \rho)$
in which
$a\vdash b = \rho(a)* d(b)$,
$a\dashv b = a * d(\rho(b)) = a * \rho(d(b))$,
for $a,b\in W$.
On the other hand, the new operation $\cdot $ on $A$ given
by $a\cdot b = a* \rho(b)$ turns $A$ into a Zinbiel algebra,
$d$ remains a derivation relative to this new operation,
and
$u\vdash v = d(v)\cdot u = v\succ u$, $u\dashv v = u\cdot d(v)= u\prec v$
for $u,v\in V$.
Therefore, we would embed a $D\Zinb $-algebra into a~differential
Zinbiel algebra which is not the case.

\subsection*{Acknowledgments}
F.~Mashurov and B.~Sartayev were supported by the Science Committee of the Ministry of Education and Science of the Republic of Kazakhstan (Grant No.~AP14870282). P.~Kolesnikov was supported by the Program of Fundamental Research RAS (project FWNF-2022-0002).
The authors are grateful to the referees for useful comments.

\pdfbookmark[1]{References}{ref}
\LastPageEnding

\end{document}